\documentclass{amsart}

\usepackage{tikz, enumerate, amsmath, amsfonts, amssymb, amsthm, wasysym, graphics, graphicx, xcolor, url, hyperref, hypcap}
\hypersetup{colorlinks=true, citecolor=darkblue, linkcolor=darkblue}
\usetikzlibrary{calc,arrows}

\newtheorem{theorem}{Theorem}[section]

\newtheorem{cor}[theorem]{Corollary}

\theoremstyle{definition}
\newtheorem{definition}[theorem]{Definition}
\newtheorem{example}[theorem]{Example}

\theoremstyle{remark}

\numberwithin{equation}{section}

\definecolor{darkblue}{rgb}{0,0,0.7} % darkblue color
\newcommand{\darkblue}{\color{darkblue}} % darkblue command
\newcommand{\defn}[1]{\textbf{\darkblue #1}} % emphasis of a definition

\newcommand{\R}{\mathbb{R}}  % The real numbers.
\newcommand{\Pe}{\mathbb{P}} % The projective space

\newcommand{\C}{\mathbb{C}} % The complex numbers
\newcommand{\dimn}{\mathrm{dim}} % dimension
\newcommand{\lra}[1]{\mathrel{\mathop{\longrightarrow}^{\mathrm{#1}}}} % Long right arrow with superscript #1
\newcommand{\lmp}[1]{\mathrel{\mathop{\longmapsto}^{\mathrm{#1}}}} % Maps arrow with superscript #1
\newcommand{\hlra}[1]{\mathrel{\mathop{\lhook\joinrel\longrightarrow}^{\mathrm{#1}}}} %Long hookrightarrow with superscript #1

\newcommand{\fiber}[1]{m_Q^{-1}(#1B/B)} % fiber of BS^Q
\newcommand{\Dem}{\text{Dem}} % Demazure product

\newcommand{\sizeQ}{m} % size of the word \sq{Q}
\newcommand{\sfr}{{\sf r}} % root
\newcommand{\sfw}{{\sf w}} % weight
\newcommand{\sfB}{{\sf B}} % Brick
\newcommand{\sort}[1]{\mathcal{N}_Q} % Sorting network of Q

\newcommand{\Root}[2]{{\sfr}(#1,#2)} % the root of #1 at position #2
\newcommand{\Weight}[2]{{\sfw}(#1,#2)} % the weight of #1 at position #2
\newcommand{\splex}[2]{\Delta(#1,#2)} % Subword complex of the word #1 and element #2
\begin{document}

%%
%% The title of the paper goes here.  Edit to your title.
%%

\title{Brick manifolds and toric varieties of brick polytopes} % realized inside Bott-Samelson varieties

%%
%% Now edit the following to give your name and address:
%% 

\author{Laura Escobar}
\address{Department of Mathematics, Cornell University, 
Ithaca, NY 14850}
%\email{le78@cornell.edu}
%\urladdr{www.math.sc.edu/$\sim$howard} % Delete if not wanted.

%%
%% If there is another author uncomment and edit the following.
%%

%\author{Second Author}
%\address{Department of Mathematics, University of South Carolina,
%Columbia, SC 29208}
%\email{second@math.sc.edu}
%\urladdr{www.math.sc.edu/$\sim$second}

%%
%% If there are three of more authors they are added in the obvious
%% way. 
%%

%%%
%%% The following is for the abstract.  The abstract is optional and
%%% if not used just delete, or comment out, the following.
%%%

\begin{abstract}
Bott-Samelson varieties are a twisted product of $\mathbb{C}\mathbb{P}^1$'s with a map into $G/B$. 
These varieties are mostly studied in the case in which the map into $G/B$ is birational to the image; however in this paper we study a fiber of this map when it is not birational. We will see that in some cases the general fiber, which we christen a brick manifold, is a toric variety. In order to do so we use the moment map of a Bott-Samelson variety to translate this problem into one in terms of the ``subword complexes" of Knutson and Miller. Pilaud and Stump realized certain subword complexes as the dual of the boundary of a polytope which generalizes the brick polytope defined by Pilaud and Santos. 
For a nice family of words, the brick polytope is the generalized associahedron realized by Hohlweg and Lange.
These stories connect in a 
nice way: the moment polytope of the brick manifold is the brick polytope.
In particular, we give a nice description of the toric variety of the associahedron.
We give each brick manifold a stratification dual to the subword complex. 
In addition, we relate brick manifolds to Brion's resolutions of Richardon varieties.

\end{abstract}

%%
%%  LaTeX will not make the title for the paper unless told to do so.
%%  This is done by uncommenting the following.
%%

 \maketitle

%%
%% LaTeX can automatically make a table of contents.  This is done by
%% uncommenting the following:
%%

\tableofcontents

%%%%%%%%%%%%%%%%%%%%%%%%%%%%%%%%%%%%%%%%%%%%%%%%%%%%%%%%%%%%%%%%%%%%%%
\section*{Introduction}
%%%%%%%%%%%%%%%%%%%%%%%%%%%%%%%%%%%%%%%%%%%%%%%%%%%%%%%%%%%%%%%%%%%%%%

The Bott-Samelson varieties were first defined by Raoul Bott and Hans Samelson in 
\cite{MR0071773}. 
Bott-Samelson varieties are a twisted product of $\mathbb{C}\mathbb{P}^1$'s with a map into $G/B$. 
These varieties have been studied mostly in the case in which the map into $G/B$ is birational. 
In this paper we study some fibers of this map when it is not birational to the image.
We show that for some Bott-Samelson varieties this fiber is a toric variety. 
In order to do so we translate this problem into a purely combinatorial one in terms of subword complexes.
These simplicial complexes $\Delta(Q,w)$ depend on a word $Q$ in the generators of the Weyl group $W$ of $G$ and an element $w\in W$.
They were defined by Allen Knutson and Ezra Miller in \cite{MR2047852} to describe the geometry of determinantal ideals and Schubert polynomials.
In \cite{PS}, Vincent Pilaud and Christian Stump defined the brick polytope and realized certain subword complexes as the boundary of a polytope dual to the brick polytope.
In \cite{CLS} Cesar Ceballos, Jean-Philippe Labb{\'e} and Stump showed that for a nice family of words, the brick polytope is the cluster polytope and for the Weyl group of type A it is an associahedron.
In Theorem \ref{toric} we prove that for the words Pilaud and Stump define as ``realizable", a fiber of the Bott-Samelson map is the toric variety of the brick polytope.
We then get a description of the toric variety of the associahedron in terms of flags arranged in a poset.

Actually the toric case is just a shadow of a more general situation.
We prove in Theorem \ref{moment} that for any word $Q$ and element $w\in W$ the brick polytope is the moment polytope of a fiber of the Bott-Samelson variety.
This motivates us to define the brick manifold as the fiber studied here.
In this paper we show a very nice connection between subword complexes, brick polytopes and brick manifolds. In Theorem \ref{toric} we classify the toric brick manifolds. We end the paper with two results about brick manifolds: we exhibit a stratification of the brick manifolds dual to the subword complex in Theorem \ref{strati} and following \cite{MR2143072}, show that brick manifolds provide resolutions for Richardson varieties in Theorem \ref{richi}.\\

\noindent \textbf{Acknowledgments.} I would like to thank my advisor Allen Knutson sharing this project with me and for all the incredibly fruitful meetings that happen at Cornell University. I also wish to thank Cesar Ceballos for discussions regarding the brick polytope and the subword complex.

%%%%%%%%%%%%%%%%%%%%%%%%%%%%%%%%%%%%%%%%%%%%%%%%%%%%%%%%%%%%%%%%%%%%%%
\section{Some definitions}
%%%%%%%%%%%%%%%%%%%%%%%%%%%%%%%%%%%%%%%%%%%%%%%%%%%%%%%%%%%%%%%%%%%%%%

%%%%%%%%%%%%%%%%%%%%%%%%%%%%%%%%%%%%
\subsection{Subword complexes}
%%%%%%%%%%%%%%%%%%%%%%%%%%%%%%%%%%%%

Let $W$ be the Weyl group of a complex Lie group $G$ with respect to a torus $T$, i.e., $W$ is a crystallographic Coxeter group, and let $S=\{s_i : i\in I\}$ denote its generators.

Some notation: Let $Q=(q_1,\ldots,q_\sizeQ)$ be a \defn{word} in $S$, i.e. an ordered sequence of elements of $S$. A \defn{subword} $J=(r_1,\ldots,r_\sizeQ)$ of $Q$ is a word obtained from $Q$ by replacing some of its letters by $-$. There are a total of $2^{|Q|}$ subwords of $Q$. Given a subword $J$, we denote by $Q\setminus J$ the subword with $k$-th entry equal to $-$ if $r_k\neq -$ and equal to $q_k$ otherwise for $k=1,\ldots,\sizeQ$. For example, $J=(s_1,-,s_3,-,s_2)$ is a subword of $Q=(s_1,s_2,s_3,s_1,s_2)$ and $Q\setminus J=(-,s_2,-,s_1,-)$. Given a subword $J$ we denote by $J_{(k)}$ the product of the leftmost $k$ letters in $J$ with $-$ behaving as the identity, if $k\geq 1$, and $J_{(0)}=1$.

\begin{definition}
 Let $Q=(q_1,\ldots,q_\sizeQ)$ be a word in $S$ and $w\in W$. The \defn{subword complex} $\splex{Q}{w}$ is the simplicial complex on the vertex set $Q$ whose facets are the subwords $F$ of $Q$ such that the product $(Q\setminus F)_{(\sizeQ)}$ is a reduced expression for $w$.
\end{definition}

\begin{example} Let $Q=(s_1,s_2,s_1,s_2,s_1)$ and $w=s_1s_2s_1$, then the simplicial complex $\splex{Q}{w}$ is
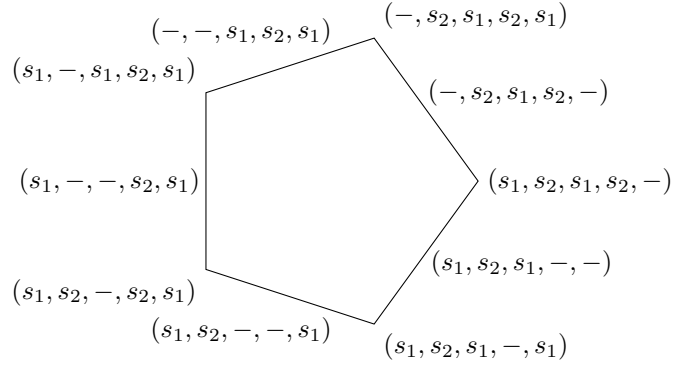
\begin{center}
\begin{tikzpicture}
  \newdimen\rad
  \rad=2cm
  \draw (0:\rad)
     \foreach \x in {72,144,...,360} {  -- (\x:\rad) }
              -- cycle (360:\rad) node[right] {$(s_1,s_2,s_1,s_2,-)$}
              -- cycle (288:\rad) node[below right] {$(s_1,s_2,s_1,-,s_1)$}
              -- cycle (216:\rad) node[below left] {$(s_1,s_2,-,s_2,s_1)$}
              -- cycle (144:\rad) node[above left] {$(s_1,-,s_1,s_2,s_1)$}
              -- cycle  (72:\rad) node[above right] {$(-,s_2,s_1,s_2,s_1)$};
	\node (facet15) at (25:2.8cm) {$(-,s_2,s_1,s_2,-)$};
	\node (facet12) at (120:2.3cm) {$(-,-,s_1,s_2,s_1)$};
	\node (facet23) at (180:2.9cm) {$(s_1,-,-,s_2,s_1)$};
	\node (facet34) at (240:2.3cm) {$(s_1,s_2,-,-,s_1)$};
	\node (facet45) at (337:2.8cm) {$(s_1,s_2,s_1,-,-)$};
\end{tikzpicture}\end{center}
In order to make the reduced expression more explicit, we are labeling the faces by their complements.
\end{example}

\begin{definition} We define the \defn{Demazure product} of a word $Q$ inductively as follows:
\begin{itemize}
 \item $\Dem(\text{empty word})=\text{id}$
 \item $\Dem((Q, s))=\begin{cases} \Dem(Q)\cdot s &\text{ if } \ell(\Dem(Q)s)>\ell(\Dem(Q)) \\
\Dem(Q) & \text{ if } \ell(\Dem(Q)s)<\ell(\Dem(Q))\end{cases}$
\end{itemize}
\end{definition}

In \cite{MR2047852} the authors prove that $\splex{Q}{w}$ is a sphere if and only if $\Dem(Q)=w$. In this paper we only consider such pairs. If in addition we assume $Q$ is reduced, then $\splex{Q}{w}=\{\emptyset\}$, the $(-1)$-sphere, so we will not consider reduced $Q$ in this paper.
%%%%%%%%%%%%%%%%%%%%%%%%%%%%%%%%%%%%
\subsection{Brick polytopes}\label{brick}
%%%%%%%%%%%%%%%%%%%%%%%%%%%%%%%%%%%%

 Let $\Delta(W):=\{\alpha_s : s\in S\}$ be the simple roots of $W$ and let $\nabla(W):=\{\omega_i : s_i\in S\}$ be its fundamental weights. Pilaud and Stump define brick polytopes and study their properties in \cite{PS}. 
 For them, the brick polytope is the convex hull of some conjugates of the fundamental weights of the Weyl group, one per each facet of the subword complex. Our definitions in this section are based on theirs, however we make the brick polytope be the convex hull of the brick vectors corresponding to all the faces in the subword complex such that the product of the complement is $w$. It turns out that the two definitions are equivalent as the proof of Theorem \ref{moment} exhibits.
 
 Given a subword complex $\splex{Q}{w}$ with $|Q|=\sizeQ$ 
define the \defn{root function} \[\Root{J}{\cdot} : \{\text{subwords of }Q\} \to \Delta(W)\]
\begin{equation}\label{other} \Root{J}{k} := (Q\setminus J)_{(k-1)}(\alpha_{q_k})\end{equation} and the \defn{weight function} \[\Weight{J}{\cdot} : \{\text{subwords of }Q\} \to \nabla(W)\]
\begin{equation}\label{otherone} \Weight{J}{k} := (Q\setminus J)_{(k-1)}(\omega_{q_k}).\end{equation}

\begin{definition}
 The \defn{brick vector} of a face $J$ of $\splex{Q}{w}$ is defined by 
 \begin{displaymath}B(J):=\sum_{k\in[\sizeQ]}\Weight{J}{k},\end{displaymath}
and the \defn{brick polytope} is the convex hull of the brick vectors of some faces of $\splex{Q}{w}$
\begin{displaymath}B(Q,w):=\text{conv}\{B(J) : J\in\splex{Q}{w} \text{ and } ( Q\setminus J)_{(\sizeQ)}=w\}.\end{displaymath}
\end{definition}

\begin{definition}\label{rootind}
 A word $Q$ is \defn{root independent} if for some vertex $B(J)$ of $B(Q,w)$ (or all vertices) we have that the multiset $r(J):=\{\{\Root{J}{i} : i\in J\}\}$ is linearly independent.
\end{definition}

Pilaud and Stump in \cite{PS} use the terminology realizing instead of root independent. They show that if $Q$ is root independent, then the brick polytope is dual to the subword complex. One of the main theorems of this paper states that the brick manifold of a word $Q$ is toric with respect to a maximal torus of the Lie group when $Q$ is root independent.

%%%%%%%%%%%%%%%%%%%%%%%%%%%%%%%%%%%%%%%%%%%%%%%%%%%%%%%%%%%%%%%%%%%%%%
\section{Brick manifolds for $SL_n(\C)$}\label{sec:bsgl}
%%%%%%%%%%%%%%%%%%%%%%%%%%%%%%%%%%%%%%%%%%%%%%%%%%%%%%%%%%%%%%%%%%%%%%

We start with the case $G=SL_n(\C)$ both because it has beautiful combinatorial pictures and as a motivation to the general complex semi-simple Lie group case.

%%%%%%%%%%%%%%%%%%%%%%%%%%%%%%%%%%%%
\subsection{Brick polytopes in the $SL_n(\C)$ case}
%%%%%%%%%%%%%%%%%%%%%%%%%%%%%%%%%%%%

The \defn{sorting network} $\sort{Q}$ of a word $Q=(q_1,\ldots,q_{\sizeQ})$ consists of $n$ horizontal lines (called the \defn{levels}) and $m$ vertical segments (called the \defn{commutators}) drawn from left to right so that each commutator joins consecutive levels, no two commutators share a common endpoint, and if $q_k=s_i$ then the $k$-th commutator connects levels $i$ and $i+1$. A \defn{brick} of $\sort{Q}$ is a connected component of its complement, bounded on the left by a commutator.

A \defn{pseudoline} supported by $\sort{Q}$ is a path on $\sort{Q}$ traveling monotonically from left to right. A commutator of $\sort{Q}$ is called a \defn{crossing} between two pseudolines if it is crossed by the two pseudolines and it is called a \defn{contact} otherwise. A \defn{pseudoline arrangement} on $\sort{Q}$ is a collection of $n$ pseudolines such that each two have at most one crossing and no other intersection.

\begin{example} Let $Q=(s_1,s_2,s_1,s_2,s_1)$ then $w_0=\Dem(Q)=s_1s_2s_1=s_2s_1s_2$. The sorting network $\sort{Q}$ is
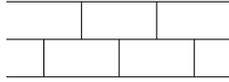
\begin{center}    
\begin{tikzpicture}
   % Define the points
   \path (0,0) coordinate (P7);
   \path (0.5,0) coordinate (P0);
   \path (1,0) coordinate (P1);
   \path (1.5,0) coordinate (P2);
   \path (2,0) coordinate (P3);
   \path (2.5,0) coordinate (P4);
   \path (3,0) coordinate (P5);
   
   \path (0,0.5) coordinate (Q7);
   \path (0.5,0.5) coordinate (Q0);
   \path (1,0.5) coordinate (Q1);
   \path (1.5,0.5) coordinate (Q2);
   \path (2,0.5) coordinate (Q3);
   \path (2.5,0.5) coordinate (Q4);
   \path (3,0.5) coordinate (Q5);

   \path (0,1) coordinate (R7);
   \path (0.5,1) coordinate (R0);
   \path (1,1) coordinate (R1);
   \path (1.5,1) coordinate (R2);
   \path (2,1) coordinate (R3);
   \path (2.5,1) coordinate (R4);
   \path (3,1) coordinate (R5);

   % Draw the edges
   \draw (P7) -- (P0) -- (P1) -- (P2) -- (P3) -- (P4) -- (P5);
   \draw (Q7) -- (Q0) -- (Q1) -- (Q2) -- (Q3) -- (Q4) -- (Q5);
   \draw (R7) -- (R0) -- (R1) -- (R2) -- (R3) -- (R4) -- (R5);
   
   \draw (Q0) -- (P0);
   \draw (Q1) -- (R1);
   \draw (P2) -- (Q2);
   \draw (Q3) -- (R3);
   \draw (P4) -- (Q4);
   
   \end{tikzpicture}
   \end{center}
    and
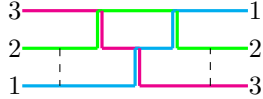
\begin{center}    
\begin{tikzpicture}
   % Define the points
     \path (0,0) coordinate (P7);
   \path (0.5,0) coordinate (P0);
   \path (1,0) coordinate (P1);
   \path (1.5,0) coordinate (P2);
   \path (2,0) coordinate (P3);
   \path (2.5,0) coordinate (P4);
   \path (3,0) coordinate (P5);
   
   \path (0,0.5) coordinate (Q7);
   \path (0.5,0.5) coordinate (Q0);
   \path (1,0.5) coordinate (Q1);
   \path (1.5,0.5) coordinate (Q2);
   \path (2,0.5) coordinate (Q3);
   \path (2.5,0.5) coordinate (Q4);
   \path (3,0.5) coordinate (Q5);

   \path (0,1) coordinate (R7);
   \path (0.5,1) coordinate (R0);
   \path (1,1) coordinate (R1);
   \path (1.5,1) coordinate (R2);
   \path (2,1) coordinate (R3);
   \path (2.5,1) coordinate (R4);
   \path (3,1) coordinate (R5);

%epsilon to the right
   
       \path (0.06,0) coordinate (P7');
   \path (0.56,0) coordinate (P0');
   \path (1.06,0) coordinate (P1');
   \path (1.56,0) coordinate (P2');
   \path (2.06,0) coordinate (P3');
   \path (2.56,0) coordinate (P4');
   \path (3.06,0) coordinate (P5');
   
   \path (0.06,0.5) coordinate (Q7');
   \path (0.56,0.5) coordinate (Q0');
   \path (1.06,0.5) coordinate (Q1');
   \path (1.56,0.5) coordinate (Q2');
   \path (2.06,0.5) coordinate (Q3');
   \path (2.56,0.5) coordinate (Q4');
   \path (3.06,0.5) coordinate (Q5');

   \path (0.06,1) coordinate (R7');
   \path (0.56,1) coordinate (R0');
   \path (1.06,1) coordinate (R1');
   \path (1.56,1) coordinate (R2');
   \path (2.06,1) coordinate (R3');
   \path (2.56,1) coordinate (R4');
   \path (3.06,1) coordinate (R5');
     
     \node (a) at (-0.1,0) {1};
      \node (b) at (-0.1,0.5) {2};
       \node (c) at (-0.1,1) {3};

     \node (a) at (3.1,0) {3};
      \node (b) at (3.1,0.5) {2};
       \node (c) at (3.1,1) {1};
      
   % Draw the edges
   {\color{cyan}\draw[very thick] (P7) -- (P0);}
   {\color{cyan}\draw[very thick] (P0) -- (P1) -- (P2);}
   {\color{magenta}\draw[very thick] (P2') -- (P3) -- (P4) -- (P5);}
   
   {\color{green}\draw[very thick] (Q7) -- (Q0) -- (Q1);}
   {\color{magenta}\draw[very thick] (Q1') -- (Q2');}
   {\color{cyan}\draw[very thick] (Q2) -- (Q3);}
   {\color{green}\draw[very thick] (Q3') -- (Q4) -- (Q5);}
   
   {\color{magenta}\draw[very thick] (R7) -- (R0) -- (R1');}
   {\color{green}\draw[very thick] (R1) -- (R2) -- (R3');}
   {\color{cyan}\draw[very thick] (R3) -- (R4) -- (R5);}
   
   \draw[dashed] (Q0) -- (P0);
   {\color{green}\draw[very thick] (Q1) -- (R1);}
   {\color{magenta}\draw[very thick] (Q1') -- (R1');}
   {\color{cyan}\draw[very thick] (P2) -- (Q2);}
   {\color{magenta}\draw[very thick] (P2') -- (Q2');}
   {\color{cyan}\draw[very thick] (Q3) -- (R3);}
      {\color{green}\draw[very thick] (Q3') -- (R3');}
   \draw[dashed] (P4) -- (Q4);
 
   \end{tikzpicture}
   \end{center}
   is a pseudoline arrangement on $\sort{Q}$.
\end{example}

Given a pseudoline arrangement supported by $\sort{Q}$, if we let $J=(r_1,\ldots,r_\sizeQ)$ be the subword of $Q$ with $r_i\neq -$ precisely when there is a contact at the $i$-th commutator, then the product $w=(Q\setminus J)_{(\sizeQ)}$ is an element of $W$ and the pseudoline ending on the right at level $i$ will start on the left at level $w(i)$. We call such an arrangement a \defn{$w$-pseudoline arrangement}. There is a one-to-one correspondence between faces $J$ of $\splex{Q}{w}$ and $\left( Q\setminus J\right)_{(\sizeQ)}$-pseudoline arrangements supported by $\sort{Q}$. The pseudoline arrangement in the previous example corresponds to the subword $J=(s_1,-,-,-,s_1)$.
 In this setup, we have that $w(J,j)$ is the characteristic vector of the pseudolines passing below the $j$-th brick of $\sort{Q}$. Moreover, the $i$-th coordinate of the \defn{brick vector} $B(J)$ is the number of bricks in $\sort{Q}$ that lie above the $i$-th pseudoline with contacts $J$, and the brick polytope $B(Q,w)$ is the following convex hull:
$$B(Q,w):=\text{conv}\{B(J) : J\in\splex{Q}{w} \text{ and } (Q\setminus J)_{(\sizeQ)}=w\}.$$

\begin{example}\label{pent}
 Let $Q=(s_1,s_2,s_1,s_2,s_1)$ and $w=s_1s_2s_1$. Then the pseudoline arrangement corresponding to the subword $J=(s_1,-,-,-,s_1)$ gives the vector $B(J)=(2,1,4)$ obtained by counting bricks above each line. The brick polytope $B(Q,w)$ is pictured below.
     
 
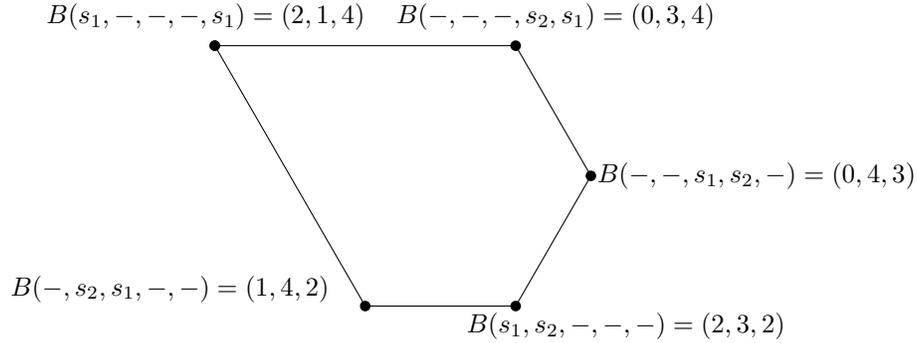
\begin{center}  
\begin{tikzpicture}
%%%  define vertices with coordinates
\coordinate (0;0) at (0,0); 
\foreach \c in {1,...,3}{%  
\foreach \i in {0,...,5}{% 
\pgfmathtruncatemacro\j{\c*\i}
\coordinate (\c;\j) at (60*\i:\c);  
} }
\foreach \i in {0,2,...,10}{% 
\pgfmathtruncatemacro\j{mod(\i+2,12)} 
\pgfmathtruncatemacro\k{\i+1}
\coordinate (2;\k) at ($(2;\i)!.5!(2;\j)$) ;}

\foreach \i in {0,3,...,15}{% 
\pgfmathtruncatemacro\j{mod(\i+3,18)} 
\pgfmathtruncatemacro\k{\i+1} 
\pgfmathtruncatemacro\l{\i+2}
\coordinate (3;\k) at ($(3;\i)!1/3!(3;\j)$)  ;
\coordinate (3;\l) at ($(3;\i)!2/3!(3;\j)$)  ;
 }

 %%%%%%%%% draw lines %%%%%%%% 
 \draw (3;7)--(2;9)--(3;16)--(3;0)--(3;2)--(3;7);
    
%%%%%%%%% draw points %%%%%%%% 
\fill (3;7) circle (2pt);
\fill (2;9) circle (2pt);
\fill (3;16) circle (2pt);
\fill (3;0) circle (2pt);
\fill (3;2) circle (2pt);
\fill (3;7) circle (2pt);

%%%%%%%% Labels %%%%%%%%%
\node (b) at (40:3.3) {$B(-,-,-,s_2,s_1)=(0,3,4)$};
\node (a) at (135:3) {$B(s_1,-,-,-,s_1)=(2,1,4)$};
\node (e) at (330:4) {$B(s_1,s_2,-,-,-)=(2,3,2)$};
\node (c) at (60*3:-5.2) {$B(-,-,s_1,s_2,-)=(0,4,3)$};
\node (d) at (210:3) {$B(-,s_2,s_1,-,-)=(1,4,2)$};
\end{tikzpicture} 
\end{center}
For more pictures of brick polytopes of various $Q$ and $w$, see \cite{PS}.\end{example}

A purpose of this paper is to assign geometry to these polytopes. To do so, we use the Bott-Samelson varieties which we define in the following section.

%%%%%%%%%%%%%%%%%%%%%%%%%%%%%%%%%%%%
\subsection{Definition of Bott-Samelson varieties for $SL_n(\C)$}\label{slnbs}
%%%%%%%%%%%%%%%%%%%%%%%%%%%%%%%%%%%%

Let $G=SL_n(\C)$ and fix an ordered basis for $\C^n$. Let $B$ be the subgroup $SL_n(\C)$ consisting of upper triangular matrices with respect to this basis. We then get an ascending flag of $B$-invariant vector spaces
\begin{displaymath}\langle e_1\rangle\subset \cdots\subset \langle e_1,\ldots,e_n\rangle,\end{displaymath}
which we refer to as the \defn{base flag}.
Let $T$ be the subgroup consisting of all diagonal matrices in $G$, so $T$ is a maximal torus contained in $B$.
Let $P_i$ be the minimal parabolic subgroup consisting of all matrices that are upper triangular except possibly at the position $(i+1,i)$. 
The quotient $G/B$ is the flag variety, that is, the space of flags $\{0\}\subset V_1\subset \cdots\subset V_n=\C^n$ where each $V_i$ is an $i$-dimensional vector space. Moreover, the Weyl group of $G$ is $W=A_{n-1}$ with generators $S=\{s_1,\ldots,s_{n-1}\}$. The fundamental weights are $\nabla(W)=\{\omega_i:i=1,\ldots,n-1\}$ where the first $i$ entries of $\omega_i$ are 1 and the rest are 0.

We begin the definition of $BS^Q$ with an example.
\begin{example}
Let $G=SL_3(\C)$ and $Q=(s_1,s_2,s_1,s_2,s_1)$. Then the Bott-Samelson variety $BS^Q$ is constructed by starting with the base flag and then iteratively reading the word from left to right: if the $k$-th letter of $Q$ is $s_i$, we have an $i$-th dimensional vector space $V_k$  such that $V_{k-1}\subset V_k\subset V_{k+1}$.
In this example we have that 
\begin{displaymath} BS^Q=\{(V_1,V_2,V_3,V_4,V_5) : \text{ the diagram below holds}\}\end{displaymath}

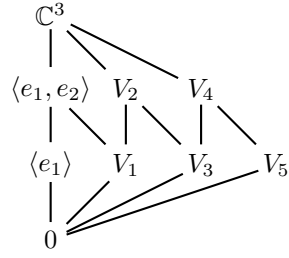
\begin{center}
  \begin{tikzpicture}
\node (top) at (0,0) {$\C^3$};
    \node [below of=top] (flag2)  {$\langle e_1,e_2\rangle$};
    \node [below of=flag2] (flag1) {$\langle e_1\rangle$};
    \node [right of=flag1] (L1) {$V_1$};
        \node [right of=flag2] (P1) {$V_2$};
        \node [right of=L1] (L2) {$V_3$};
\node [right of=P1] (P2) {$V_4$};
\node [right of=L2] (L3) {$V_5$};
\node [below of=flag1] (down) {0};
\draw [thick,shorten <=-2pt] (top) -- (flag2);
\draw [thick,shorten <=-2pt] (flag2) -- (flag1);
\draw [thick,shorten <=-2pt] (P1) -- (L1);
\draw [thick,shorten <=-2pt] (P1) -- (L2);
\draw [thick,shorten <=-2pt] (P2) -- (L2);
\draw [thick,shorten <=-2pt] (P2) -- (L3);
\draw [thick,shorten <=-2pt] (top) -- (P1);
\draw [thick,shorten <=-2pt] (top) -- (P2);
\draw [thick,shorten <=-2pt] (flag2) -- (L1);
\draw [thick,shorten <=-2pt] (flag1) -- (down);
\draw [thick,shorten <=-2pt] (L1) -- (down);
\draw [thick,shorten <=-2pt] (L2) -- (down);
\draw [thick,shorten <=-2pt] (L3) -- (down);
\end{tikzpicture}
\end{center}
\end{example}

More generally, if $Q=(q_1,\ldots,q_m)$ then $BS^Q$ consists of a list of $m+1$ flags where the zeroth one is the base flag and such that the $k$-th one agrees with the previous one except possibly on the $k$-th subspace $V_k$. We can give a point in $BS^Q$ by giving the subspaces $(V_1,\ldots,V_m)$ such that the incidence relations given by the flags hold. This carries a $B$-action, and the map $\displaystyle{BS^Q\lra{m_Q}G/B}$ mapping the list to the last flag is $B$-equivariant.

\begin{example} Continuing with the previous example, we have that \[m_Q:BS^{(s_1,s_2,s_1,s_2,s_1)}\rightarrow G/B\] is the map
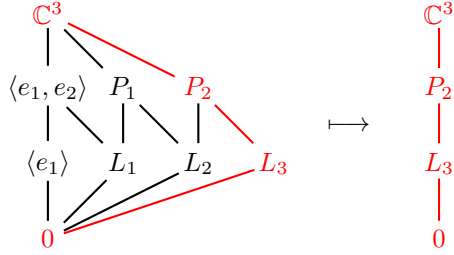
\begin{center}
  \begin{tikzpicture}
    {\color{red}\node (top) at (0,0) {$\C^3$};}
    \node [below of=top] (flag2)  {$\langle e_1,e_2\rangle$};
    \node [below of=flag2] (flag1) {$\langle e_1\rangle$};
    \node [right of=flag1] (L1) {$L_1$};
        \node [right of=flag2] (P1) {$P_1$};
        \node [right of=L1] (L2) {$L_2$};
        {\color{red}\node [right of=P1] (P2) {$P_2$};}
        {\color{red}\node [right of=L2] (L3) {$L_3$};}
        {\color{red}\node [below of=flag1] (down) {0};}
\draw [thick,shorten <=-2pt] (top) -- (flag2);
\draw [thick,shorten <=-2pt] (flag2) -- (flag1);
\draw [thick,shorten <=-2pt] (P1) -- (L1);
\draw [thick,shorten <=-2pt] (P1) -- (L2);
\draw [thick,shorten <=-2pt] (P2) -- (L2);
{\color{red}\draw [thick,shorten <=-2pt] (P2) -- (L3);}
\draw [thick,shorten <=-2pt] (top) -- (P1);
{\color{red}\draw [thick,shorten <=-2pt] (top) -- (P2);}
\draw [thick,shorten <=-2pt] (flag2) -- (L1);
\draw [thick,shorten <=-2pt] (flag1) -- (down);
\draw [thick,shorten <=-2pt] (L1) -- (down);
\draw [thick,shorten <=-2pt] (L2) -- (down);
{\color{red}\draw [thick,shorten <=-2pt] (L3) -- (down);}

\node (map) at (4,-1.5) {$\longmapsto$};

        {\color{red}    \node (mtop) at (5.2,0) {$\C^3$};}
        {\color{red}    \node [below of=mtop] (mP2)  {$P_2$};}
        {\color{red}    \node [below of=mP2] (mL3) {$L_3$};}
        {\color{red}    \node [below of=mL3] (mdown) {0};}
        {\color{red}\draw [thick,shorten <=-2pt] (mtop) -- (mP2);}
        {\color{red}\draw [thick,shorten <=-2pt] (mP2) -- (mL3);}
        {\color{red}\draw [thick,shorten <=-2pt] (mL3) -- (mdown);}
\end{tikzpicture}  
\end{center}
\end{example}

We now define the main object of study in this paper.
\begin{definition} Let $Q=(q_1,\ldots,q_\sizeQ)$ be a word in the generators of $W$ and $w=\Dem(Q)$, then the \defn{brick manifold} is the fiber $\fiber{w}$. 
\end{definition}

Note that the $B$-action restricted to $T$ is just the extension to $BS^Q$ of $T$ acting on $\C^n$ by multiplication. There is a 1-1 correspondence between $T$-fixed points on $BS^Q$ and subwords $J$ of $Q$ such that if $p(J)$ is the $T$-fixed point corresponding to $J$ then $m_Q(p(J))=\left( Q\setminus J\right)_{(\sizeQ)}B/B\in G/B$. The point \defn{$p(J)$ corresponding to the subword $J=(r_1,\ldots,r_\sizeQ)$} is determined by deciding between $=$ and $\neq$ in each diamond
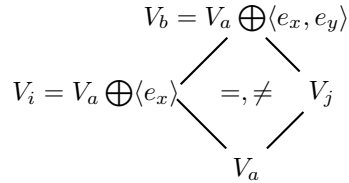
\begin{center}
 \begin{tikzpicture}
    \node (top) at (0,0) {$V_b=V_a\bigoplus\langle e_x,e_y \rangle$};
    \node [below of=top] (nothing)  {$=,\neq$};
    \node[left of=nothing] (new) {};
    \node [left of=new] (left) {$V_i=V_a\bigoplus\langle e_x\rangle$};
    \node [right of=nothing] (right) {$V_j$};
        \node [below of=nothing] (down) {$V_a$};
\draw [thick,shorten <=-2pt] (top) -- (new);
\draw [thick,shorten <=-2pt] (top) -- (right);
\draw [thick,shorten <=-2pt] (right) -- (down);
\draw [thick,shorten <=-2pt] (new) -- (down);
\end{tikzpicture}
\end{center}
using the rule: for $Q=(q_1,\ldots,q_m)$, we pick ``$=$" if $r_j=q_j$ and ``$\neq$" if $r_j=-$. We illustrate this correspondence by an example.

\begin{example} 
 The subword $J=(-,s_2,-,-,s_1)$ of $Q=(s_1,s_2,s_1,s_2,s_1)$ corresponds to the coordinate flags
   
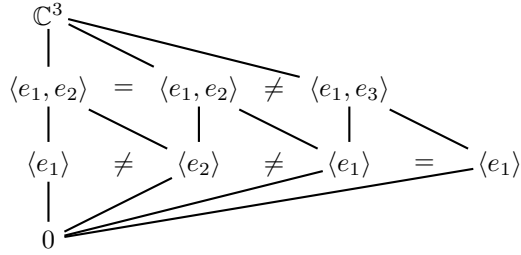
\begin{center}
   \begin{tikzpicture}
    \node (top) at (0,0) {$\C^3$};
    \node [below of=top] (flag2)  {$\langle e_1,e_2\rangle$};
    \node [below of=flag2] (flag1) {$\langle e_1\rangle$};
          \node [right of=flag1](1) {$\neq$};
                \node [right of=flag2](2) {$=$};
    \node [right of=1] (L1) {$\langle e_2\rangle$};
        \node [right of=2] (P1) {$\langle e_1,e_2\rangle$};
                \node[right of=L1] (3) {$\neq$};
        \node [right of=P1](4)  {$\neq$};
        \node [right of=3] (L2) {$\langle e_1\rangle$};
        \node [right of=4] (P2) {$\langle e_1,e_3\rangle$};
        \node [right of=L2] (5) {$=$}; 
        \node [right of=5] (L3) {$\langle e_1\rangle$};
        \node [below of=flag1] (down) {0};
\draw [thick,shorten <=-2pt] (top) -- (flag2);
\draw [thick,shorten <=-2pt] (flag2) -- (flag1);
\draw [thick,shorten <=-2pt] (P1) -- (L1);
\draw [thick,shorten <=-2pt] (P1) -- (L2);
\draw [thick,shorten <=-2pt] (P2) -- (L2);
\draw [thick,shorten <=-2pt] (P2) -- (L3);
\draw [thick,shorten <=-2pt] (top) -- (P1);
\draw [thick,shorten <=-2pt] (top) -- (P2);
\draw [thick,shorten <=-2pt] (flag2) -- (L1);
\draw [thick,shorten <=-2pt] (flag1) -- (down);
\draw [thick,shorten <=-2pt] (L1) -- (down);
\draw [thick,shorten <=-2pt] (L2) -- (down);
\draw [thick,shorten <=-2pt] (L3) -- (down);
\end{tikzpicture}  
\end{center}
and its image under $m_Q:BS^Q\rightarrow G/B$ is $(Q\setminus J)_{(\sizeQ)}B=(s_1s_1s_2)B=(s_2)B$.
  \end{example}

This correspondence motivates the relation between fibers of the map $m_Q$ and subword complexes. The main tool connecting brick polytopes with fibers of Bott-Samelson varieties will be moment maps of symplectic manifolds. We will discuss the symplectic manifold structure on general $BS^Q$ in Section \ref{symp}. Namely, we will show that Bott-Samelson varieties are Hamiltonian symplectic manifolds with respect to the torus action described above. Therefore, a Bott-Samelson variety comes equipped with a \defn{moment map} associated to the torus action. The image of this map is the \defn{moment polytope} and it equals the convex hull of the images of the $T$-fixed points. Every toric variety is a Hamiltonian symplectic manifold with respect to the torus action. Moreover, if $X$ is the toric variety associated to a Delzant polytope $P$ then the image of the moment map is the polytope $P$. 

In order to motivate latter sections and, more importantly, to be able to state the theorem connecting Bott-Samelson varieties and brick polytopes, we now describe the moment map of $BS^Q$ for the current case of interest, $G=SL_n(\C)$. The moment map is a map
\begin{displaymath}\mu: BS^Q\longrightarrow \R\langle\nabla(W)\rangle,\end{displaymath}
 where $\R\langle\nabla(W)\rangle$ is the real span of the fundamental weights of $W$. Let $\pi_V:\C^n\rightarrow V$ denote the orthogonal projection onto $V$ and let $P_V$ be the corresponding matrix with respect to the basis $e_1,\ldots,e_n$. Given $p=(V_1,\ldots,V_m)\in BS^Q$ the moment map is 
\begin{align*}
BS^Q &\lra{\mu} \R^n \\
 (V_1,\ldots,V_m) &\lmp{\mu}\sum_{i=1}^m\text{diag}(P_{V_i}).
 \end{align*}

In the following section we give a precise statement about the relation between brick polytopes and Bott-Samelson varieties.

%%%%%%%%%%%%%%%%%%%%%%%%%%%%%%%%%%%%
\subsection{Toric varieties for brick polytopes in the $SL_n(\C)$ case}
%%%%%%%%%%%%%%%%%%%%%%%%%%%%%%%%%%%%

Recall from section \ref{slnbs} that subwords $J$ of $Q$ are in bijective correspondence with $T$-fixed points of $BS^Q$, and that if $p(J)$ is the point corresponding to $J$, as defined in the previous section, then $m_Q(p(J))=\left(Q\setminus J\right)_{(\sizeQ)}B/B\in G/B$, where $\sizeQ=|Q|$.
 This means that the rightmost flag of the configuration $p(J)$ is the flag corresponding to $(Q\setminus J)_{(\sizeQ)}\in W$ and so the pseudoline arrangement corresponding to $S$ is an $\left(Q\setminus S\right)_{(\sizeQ)}$-arrangement. The following example shows the correspondence.

\begin{example}
The pseudoline arrangement corresponding to the subword $J=(s_1,-,-,-,s_1)$ gives a $T$-fixed point of $BS^{(s_1,s_2,s_1,s_2,s_1)}$. The diagram below exhibits this correspondence. Each brick of the sorting network corresponds to a coordinate subspace of a point in the Bott-Samelson variety. Given a pseudoline arrangement supported in the sorting network of $Q$, the $j$-th subspace corresponding to the $j$-th brick is the coordinate subspace spanned by the $e_i$ where $i$ ranges over those pseudolines passing below the $j$-th brick. Note then that two bricks share a crossing if and only if the corresponding coordinate spaces are equal. This will be proven in the theorem that follows.

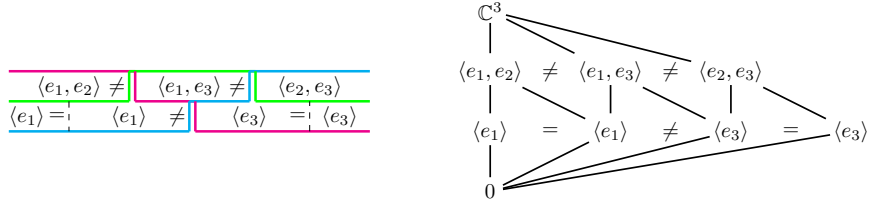
\scalebox{0.8}{
\begin{tikzpicture}
   % Define the points
     \path (0,-2) coordinate (P7);
   \path (1,-2) coordinate (P0);
   \path (2,-2) coordinate (P1);
   \path (3,-2) coordinate (P2);
   \path (4,-2) coordinate (P3);
   \path (5,-2) coordinate (P4);
   \path (6,-2) coordinate (P5);
   
   \path (0,-1.5) coordinate (Q7);
   \path (1,-1.5) coordinate (Q0);
   \path (2,-1.5) coordinate (Q1);
   \path (3,-1.5) coordinate (Q2);
   \path (4,-1.5) coordinate (Q3);
   \path (5,-1.5) coordinate (Q4);
   \path (6,-1.5) coordinate (Q5);

   \path (0,-1) coordinate (R7);
   \path (1,-1) coordinate (R0);
   \path (2,-1) coordinate (R1);
   \path (3,-1) coordinate (R2);
   \path (4,-1) coordinate (R3);
   \path (5,-1) coordinate (R4);
   \path (6,-1) coordinate (R5);
   
   %epsilon to the right
   
    \path (0.1,-2) coordinate (P7');
   \path (1.1,-2) coordinate (P0');
   \path (2.1,-2) coordinate (P1');
   \path (3.1,-2) coordinate (P2');
   \path (4.1,-2) coordinate (P3');
   \path (5.1,-2) coordinate (P4');
   \path (6.1,-2) coordinate (P5');
   
   \path (0.1,-1.5) coordinate (Q7');
   \path (1.1,-1.5) coordinate (Q0');
   \path (2.1,-1.5) coordinate (Q1');
   \path (3.1,-1.5) coordinate (Q2');
   \path (4.1,-1.5) coordinate (Q3');
   \path (5.1,-1.5) coordinate (Q4');
   \path (6.1,-1.5) coordinate (Q5');

   \path (0.1,-1) coordinate (R7');
   \path (1.1,-1) coordinate (R0');
   \path (2.1,-1) coordinate (R1');
   \path (3.1,-1) coordinate (R2');
   \path (4.1,-1) coordinate (R3');
   \path (5.1,-1) coordinate (R4');
   \path (6.1,-1) coordinate (R5');
         
   % Draw the edges
   {\color{cyan}\draw[very thick] (P7) -- (P0);}
   {\color{cyan}\draw[very thick] (P0) -- (P1) -- (P2);}
   {\color{magenta}\draw[very thick] (P2') -- (P3) -- (P4) -- (P5);}
   
   {\color{green}\draw[very thick] (Q7) -- (Q0) -- (Q1);}
   {\color{magenta}\draw[very thick] (Q1') -- (Q2');}
   {\color{cyan}\draw[very thick] (Q2) -- (Q3);}
   {\color{green}\draw[very thick] (Q3') -- (Q4) -- (Q5);}
   
   {\color{magenta}\draw[very thick] (R7) -- (R0) -- (R1');}
   {\color{green}\draw[very thick] (R1) -- (R2) -- (R3');}
   {\color{cyan}\draw[very thick] (R3) -- (R4) -- (R5);}
   
   \draw[dashed] (Q0) -- (P0);
   {\color{green}\draw[very thick] (Q1) -- (R1);}
   {\color{magenta}\draw[very thick] (Q1') -- (R1');}
   {\color{cyan}\draw[very thick] (P2) -- (Q2);}
   {\color{magenta}\draw[very thick] (P2') -- (Q2');}
   {\color{cyan}\draw[very thick] (Q3) -- (R3);}
      {\color{green}\draw[very thick] (Q3') -- (R3');}
   \draw[dashed] (P4) -- (Q4);

 %filling of the spaces
\node at (0.8,-1.75) {$=$};
\node at (4.8,-1.75) {$=$};
\node at (2.8,-1.75) {$\neq$};

\node at (0.3,-1.75) {$\langle e_1 \rangle$};
\node at (2,-1.75) {$\langle e_1 \rangle$};
\node at (4,-1.75) {$\langle e_3 \rangle$};
\node at (5.5,-1.75) {$\langle e_3 \rangle$};

\node at (1.8,-1.25) {$\neq$};
\node at (3.8,-1.25) {$\neq$};

\node at (1,-1.25) {$\langle e_1,e_2 \rangle$};
\node at (3,-1.25) {$\langle e_1,e_3 \rangle$};
\node at (5,-1.25) {$\langle e_2,e_3 \rangle$};

%Bott-Samelson
    \node (top) at (8,0) {$\C^3$};
    \node [below of=top] (flag2)  {$\langle e_1,e_2\rangle$};
    \node [below of=flag2] (flag1) {$\langle e_1\rangle$};
          \node [right of=flag1](1) {$=$};
                \node [right of=flag2](2) {$\neq$};
    \node [right of=1] (L1) {$\langle e_1\rangle$};
        \node [right of=2] (P1) {$\langle e_1,e_3\rangle$};
                \node[right of=L1] (3) {$\neq$};
        \node [right of=P1](4)  {$\neq$};
        \node [right of=3] (L2) {$\langle e_3\rangle$};
        \node [right of=4] (P2) {$\langle e_2,e_3\rangle$};
        \node [right of=L2] (5) {$=$}; 
        \node [right of=5] (L3) {$\langle e_3\rangle$};
        \node [below of=flag1] (down) {0};
\draw [thick,shorten <=-2pt] (top) -- (flag2);
\draw [thick,shorten <=-2pt] (flag2) -- (flag1);
\draw [thick,shorten <=-2pt] (P1) -- (L1);
\draw [thick,shorten <=-2pt] (P1) -- (L2);
\draw [thick,shorten <=-2pt] (P2) -- (L2);
\draw [thick,shorten <=-2pt] (P2) -- (L3);
\draw [thick,shorten <=-2pt] (top) -- (P1);x
\draw [thick,shorten <=-2pt] (top) -- (P2);
\draw [thick,shorten <=-2pt] (flag2) -- (L1);
\draw [thick,shorten <=-2pt] (flag1) -- (down);
\draw [thick,shorten <=-2pt] (L1) -- (down);
\draw [thick,shorten <=-2pt] (L2) -- (down);
\draw [thick,shorten <=-2pt] (L3) -- (down);
\end{tikzpicture}}
\end{example}

\begin{theorem} Suppose $\Dem(Q)=w$. There is a bijective correspondence between $w$-pseudoline arrangements supported by $\sort{Q}$ and $T$-fixed points of $\fiber{w}$.
 Moreover, this correspondence makes the composite map
\begin{displaymath} \fiber{w}^T\hookrightarrow \fiber{w}\lra{\mu}\R^n\end{displaymath}
be equivalent to the mapping
\begin{displaymath}B:\{w\text{-pseudoline arrangements supported by } \sort{Q}\}\longrightarrow\R^n\end{displaymath}
given in \cite{MR2864447}.
\end{theorem}

\begin{proof}
The first part of the proposition is proven in the paragraph preceding the example above. We prove the second part of this theorem using induction on $|Q|=\sizeQ$ to prove that $\mu(p(J))=\sfB(J)$ for all subwords $J$, where $p(J)=(V_1,\ldots,V_\sizeQ)$ is the point in $BS^Q$ corresponding to $J$. Let $Q=(q_1,\ldots,q_{\sizeQ+1})$. Recall that the rightmost flag of the fixed point $p(J)$ corresponding to the subword $J$ is 
\begin{center}
   \begin{tikzpicture}
    \node (top) at (0,0) {$\langle e_{w(1)},\ldots,e_{w(n)}\rangle=\C^n$};
    \node [below of=top] (mid) {$\vdots$};
    \node [below of=mid] (flag2)  {$\langle e_{w(1)},e_{w(2)}\rangle$};
    \node [below of=flag2] (flag1) {$\langle e_{w(1)}\rangle$};
        \node [below of=flag1] (down) {0};
\draw [thick,shorten <=-2pt] (top) -- (mid);
\draw [thick,shorten <=-2pt] (mid) -- (flag2);
\draw [thick,shorten <=-2pt] (flag2) -- (flag1);
\draw [thick,shorten <=-2pt] (flag1) -- (down);
\end{tikzpicture}  
\end{center}
where $w=(Q\setminus J)_{(\sizeQ+1)}$. Let $J$ be a subword of $Q$ and consider the words $Q'=(q_1,\ldots,q_\sizeQ)$ and $J'=(j_1,\ldots,j_\sizeQ)$. By induction we have that $\mu(p(J'))=\sfB(J')$. Now notice that 
\begin{align*}\mu(p(J))&=\mu(p(J'))+(\dimn_{e_1}(V_{k+1}),\ldots,\dimn_{e_n}(V_{k+1}))\\
&=\mu(p(J'))+w\cdot(1,\ldots,1,0,\ldots,0),\end{align*}
where the 0-1 vector has as many ones as $\dimn(V_{k+1})$. The vector $w\cdot(1,\ldots,1,0,\ldots,0)$ adds one to the $i$-th coordinate if and only if the brick corresponding to the commutator $q_{k+1}$ is above the $i$-th pseudoline.
\end{proof}

\begin{theorem} Let $w=\Dem(Q)$. The fiber $\fiber{w}$ is a toric variety with respect to the torus $T$ if and only if $Q$ is root independent and $\ell(w)<|Q|\leq\ell(w)+\dimn(T)$. Moreover, $\fiber{w}$ is the toric variety associated to the polytope $B(Q,w)$.
\end{theorem}

We have proved the if part of this theorem; however the only if part will follow from Theorem \ref{toric}. 
The following corollary follows from the work of Pilaud and Santos in \cite{MR2864447}. 
We define a \defn{Coxeter element} $c$ to be the product of all simple reflections in some order using each reflection only once. 
Define the \defn{$\textbf{c}$-sorting word} of $w$ to be the lexicographically first subword of $\textbf{c}^{\infty}$ that is a reduced expression for $w$.

\begin{cor}
 If $Q$ is the concatenation of a word $\textbf{c}$ representing a Coxeter element $c$ and the $\textbf{c}$-sorting word for $w_0$, then $\fiber{w_0}$ is the toric variety of the associahedron as realized in \cite{MR2321739} and in \cite{MR2864447}.
\end{cor}

\begin{example} The toric variety of the pentagon from example \ref{pent}, i.e. the associahedron corresponding to the Coxeter element $c=(s_1,s_2)$, is 
\begin{displaymath} \fiber{w}=\{(V_1,V_2,V_3) :\text{ the diagram below holds}\}\end{displaymath}
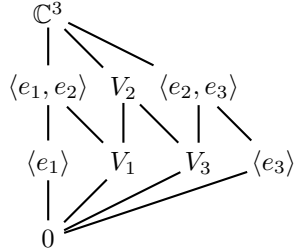
\begin{center}
  \begin{tikzpicture}
\node (top) at (0,0) {$\C^3$};
    \node [below of=top] (flag2)  {$\langle e_1,e_2\rangle$};
    \node [below of=flag2] (flag1) {$\langle e_1\rangle$};
    \node [right of=flag1] (L1) {$V_1$};
        \node [right of=flag2] (P1) {$V_2$};
        \node [right of=L1] (L2) {$V_3$};
\node [right of=P1] (P2) {$\langle e_2,e_3 \rangle$};
\node [right of=L2] (L3) {$\langle e_3 \rangle$};
\node [below of=flag1] (down) {0};
\draw [thick,shorten <=-2pt] (top) -- (flag2);
\draw [thick,shorten <=-2pt] (flag2) -- (flag1);
\draw [thick,shorten <=-2pt] (P1) -- (L1);
\draw [thick,shorten <=-2pt] (P1) -- (L2);
\draw [thick,shorten <=-2pt] (P2) -- (L2);
\draw [thick,shorten <=-2pt] (P2) -- (L3);
\draw [thick,shorten <=-2pt] (top) -- (P1);
\draw [thick,shorten <=-2pt] (top) -- (P2);
\draw [thick,shorten <=-2pt] (flag2) -- (L1);
\draw [thick,shorten <=-2pt] (flag1) -- (down);
\draw [thick,shorten <=-2pt] (L1) -- (down);
\draw [thick,shorten <=-2pt] (L2) -- (down);
\draw [thick,shorten <=-2pt] (L3) -- (down);
\end{tikzpicture}
\end{center}

\end{example}

%%%%%%%%%%%%%%%%%%%%%%%%%%%%%%%%%%%%%%%%%%%%%%%%%%%%%%%%%%%%%%%%%%%%%%
\section{Brick manifolds in the general case}\label{sec:gbs}
%%%%%%%%%%%%%%%%%%%%%%%%%%%%%%%%%%%%%%%%%%%%%%%%%%%%%%%%%%%%%%%%%%%%%%

Let $G$ be a complex semisimple Lie group, let $B$ be a Borel subgroup of $G$, i.e., a maximal solvable subgroup, and $T$ be the maximal torus contained in $B$. Let $W$ be the Weyl group of $G$ with generators $S=\{s_1,\ldots,s_n\}$, which correspond to the simple roots $\Delta(W)=\{\alpha_1,\ldots,\alpha_n\}$. Let $P$ be a parabolic subgroup of $G$, i.e., a subgroup of $G$ for which the quotient $B/P$ is a projective algebraic variety; this condition is equivalent to $P$ contains $B$. We denote by $P_i$ the minimal parabolic subgroup corresponding to $s_i$, we then have that $P_i/B\cong \C\Pe^1$. The torus $T$ acts on this quotient and this action has exactly two $T$-fixed points: one corresponding to the identity element and one corresponding to the generator $s_i$.

\begin{definition} Let $Q=(s_{i_1},\ldots,s_{i_m})$ be a word in the generators of $W$. Then the product $P_{i_1}\times \cdots\times P_{i_m}$ has an action of $B^m$ given by:
	\begin{displaymath}(b_1,\ldots,b_m)\cdot(p_1,\ldots,p_m)=(p_1b_1,b_1^{-1}p_2b_2,\ldots,b_{m-1}^{-1}p_mb_m)\end{displaymath}
 The \defn{Bott-Samelson variety} of $Q$ is the quotient of the product of the $P_i$'s by this action
\begin{displaymath}BS^Q:=(P_{i_1}\times\cdots\times P_{i_m})/B^m.\end{displaymath}
\end{definition}

Bott-Samelson varieties are smooth, irreducible and $|Q|$-dimensional algebraic varieties. They have a $B$ action given by
\begin{displaymath}b\cdot(p_1,p_2,\ldots,p_m)=(b\cdot p_1, p_2,\ldots, p_m).\end{displaymath} 
and they come equipped with a natural $B$-equivariant map 
\begin{align*} BS^Q &\lra{m_Q} G/B\\
(p_1,\ldots,p_m) &\lmp{}(p_1\cdots p_m)B/B. 
\end{align*}
The image of this map is the opposite Schubert variety $X^w:=\overline{BwB/B}$, where $w=\Dem(Q)$. In the case in which $Q$ is reduced, this map is a resolution of singularities for $X^w$, however in this paper we will study cases in which $Q$ is not reduced.

\begin{definition} Let $Q=(q_1,\ldots,q_\sizeQ)$ be a word in the generators of $W$ and $w=\Dem(Q)$, then the \defn{brick manifold} is the fiber $\fiber{w}$. 
\end{definition}

\begin{theorem} Brick manifolds are smooth, irreducible and $\dimn(\fiber{w})=|Q|-\ell(w)$.
\end{theorem}

\begin{proof} We can write the fiber as the fibered product $(wB/B)\times_{X^w}BS^Q$, so by Kleiman's transversality theorem, see \cite{MR0360616}, we have that this fiber is a smooth variety of the desired dimension. Let $N$ be the unipotent subgroup corresponding to $B$ and $N_-$ the opposite unipotent subgroup. A consequence of the Bruhat decomposition of $G/B$ is that if $N_w:=N\cap wN_-w^{-1}$, then $N_w\cdot wB/B$ is a free dense orbit in $X^w$. Since $BS^Q$ maps $B$-equivariantly to $X^w$, the preimage of $N_w\cdot wB/B$ is isomorphic to $\fiber{w}\times N_w$. Since $BS^Q$ is irreducible, it follows that the brick manifold is irreducible.
\end{proof}

%%%%%%%%%%%%%%%%%%%%%%%%%%%%%%%%%%%%
\subsection{Symplectic structure on Bott-Samelson varieties and brick manifolds}\label{symp}
%%%%%%%%%%%%%%%%%%%%%%%%%%%%%%%%%%%%
A reference for toric moment maps of coadjoint orbits is Chapter 5 of \cite{MR1414677}. Let $P_{\hat{i}}$ be the maximal parabolic subgroup of $G$ corresponding to the generators $S_{\hat{i}}:=\{s_1,\ldots,\hat{s_i},\ldots,s_n\}$. Note that for $G=SL_n(\C)$ each quotient $G/P_{\hat{i}}$ is a Grassmannian. Let $K$ be the maximal compact subgroup of $G$. Then we can view $G/P_{\hat{i}}$ as a coadjoint orbit, i.e., a $K$-orbit through the fundamental weight $\omega_i\in\mathfrak{k}^*$, where $\mathfrak{k}$ is the Lie algebra of $K$. This interpretation gives us a symplectic structure on $G/P_{\hat{i}}$ with respect to the action of $K$ such that the inclusion 
\begin{displaymath}G/P_{\hat{i}}\hlra{}\mathfrak{k}^*\end{displaymath}
is a moment map for the $K$-action. Then the composition 
\begin{displaymath}G/P_{\hat{i}}\hlra{}\mathfrak{k}^*\lra{}\mathfrak{t}^*\end{displaymath}
is the moment map of $G/P_{\hat{i}}$ with respect to the torus action, where $\mathfrak{t}$ is the Lie algebra of the torus. Moreover, the moment map for the diagonal $T$-action on a product $\prod G/P_{\hat{i}}$ is the sum of the moment maps $G/P_{\hat{i}}\lra{}\mathfrak{t}^*$.

Let $T$ act on $BS^Q$ by 
\begin{displaymath}t\cdot(p_1,p_2,\ldots,p_m)=(t\cdot p_1, p_2,\ldots, p_m).\end{displaymath} 
Given $Q=(q_1,\ldots,q_m)$ we have a $T$-equivariant inclusion
\begin{displaymath} BS^Q\hlra{\varphi} \prod_{i : s_i\in Q} G/P_{\hat{i}}\end{displaymath}
where $\varphi=(\varphi_1,\ldots,\varphi_m)$ and the $k$-th component is
	\begin{align*} BS^Q & \lra{\varphi_k} G/P_{\hat{k}} \\
	(p_1,\ldots,p_m) & \lmp{} (\prod_{i<j} p_i) P_{\hat{k}}.
	\end{align*} 
This map makes $BS^Q$ a symplectic submanifold. The composition 
	\begin{displaymath} BS^Q\hlra{\varphi} \prod_{i : s_i\in Q} G/P_{\hat{i}}\lra{} \mathfrak{t}^*\end{displaymath}
gives us a moment map for this Bott-Samelson variety with respect to the $T$-action. Thus Bott-Samelson varieties are Hamiltonian symplectic manifolds with respect to this torus action. The image of this map is the \defn{moment polytope} and by Atiyah \cite{MR642416}, Guillemin-Sternberg \cite{MR664117}, it equals the convex hull of the images of the $T$-fixed points. Recall the correspondence between $T$-fixed points on $BS^Q$ and subwords $J$ of $Q$: if $p(J)$ is the $T$-fixed point corresponding to $J$ then
\begin{displaymath}m_Q(p(J))=\left( Q \setminus J \right)_{(\sizeQ)}B/B\in G/B.\end{displaymath}
 This correspondence motivates the relation between fibers of the map $m_Q:BS^Q\longrightarrow G/B$ and subword complexes.

We now describe the image of the $T$-fixed points under the moment map. For each $k$ we have the moment map 
	\begin{displaymath}\mu_k:G/P_{\hat{k}}\lra{}\mathfrak{t}^*,\end{displaymath}
 where $\mu_k(P_{\hat{k}})=\omega_k$, the fundamental weight corresponding to $s_k$, and it maps a general element to a Weyl conjugate of this fundamental weight. Before we finish describing the maps $\mu_k$, we note that the moment map of $BS^Q$ is then 
 	\begin{displaymath}\sum_{k=1}^m\varphi_k\circ\mu_k.\end{displaymath}
 Consider the fixed point $(p_1,\ldots,p_m)$ in $BS^Q$ corresponding to the subword $J$ of $Q$ then under the moment map $\mu_k$ each $p_j$ corresponds to either the reflection $s_{i_j}$ if $q_j\in J$ or to the identity in $W$. In other words, $p_j$ corresponds to $s_{i_j}$ if $p_j\notin B$ and to the identity in $W$ otherwise. In conclusion we have that for $J$ subword of $ Q$ and 
 	\begin{displaymath}p_J= \text{ the fixed point corresponding to } J\end{displaymath}
\begin{align*} BS^Q &\lra{\varphi_k\circ\mu_k} \mathfrak{t}^*\\
p_J &\lmp{} (J )_{(k-1)}(\omega_k).
\end{align*}
It then follows that
\begin{align} BS^Q &\lra{\mu} \mathfrak{t}^*\\
\label{eq} p_J &\lmp{} \sum_{k=1}^m (J )_{(k-1)}(\omega_k)
\end{align} 

%%%%%%%%%%%%%%%%%%%%%%%%%%%%%%%%%%%%%%%%%%%%%%%%%%%%%%
\subsection{Moment polytopes of brick manifolds}
%%%%%%%%%%%%%%%%%%%%%%%%%%%%%%%%%%%%%%%%%%%%%%%%%%%%%%

We now state and prove the main results of the paper.

\begin{theorem}\label{moment} Let $w=\Dem(Q)$.
 The image of $\fiber{w}$ under the moment map is the brick polytope $B(Q,w)$.
\end{theorem}

\begin{proof}
 $T$-fixed points of $BS^Q$ are in 1-1 correspondence with subwords $J$ of $ Q$. This induces a 1-1 correspondence between $T$-fixed points of $\fiber{w}$ and the subwords $J$ of $Q$ with $(Q\setminus J)_{(\sizeQ)}=w$, where $w=\Dem(Q)$. If the subword $J$ is not a facet of the subword complex $\Delta(Q,w)$ then it gives a non reduced product $( Q\setminus J)_{(\sizeQ)}$. This implies that the root configuration $r(J)=\{\{r(J,i): i\in F\}\}$ has a smaller dimension than the root configuration of a facet and thus it cannot be a vertex. Therefore, the moment polytope is the convex hull of the points corresponding to facets of $\Delta(Q,w)$ and by Equation \ref{eq} the image of each fixed point is precisely the one defined in Equation \ref{other} in Section \ref{brick} by Pilaud and Stump.
\end{proof}

Note that this theorem does not assume that the fiber is a toric variety so the relation between brick polytopes and brick manifolds is quite strong. The following theorem classifies toric brick manifolds.

\begin{theorem}\label{toric} Let $w=\Dem(Q)$. The fiber $\fiber{w}$ is a toric variety with respect to the torus $T$ if and only if $Q$ is root independent and $\ell(w)<|Q|\leq\ell(w))+\dimn(T)$. Moreover, $\fiber{w}$ is the toric variety associated to the polytope $B(Q,w)$.
\end{theorem}

\begin{proof} Note that $\dimn(\fiber{w}) \leq \dimn(T)$. However, if we have $<$ then we can make the torus smaller and so without loss of generality we can assume the dimensions are equal. It suffices to show that $T$ doesn't have generic stabilizer of positive dimension. This is true if and only if $\mu(\fiber{w})$ spans $\R^n$ and this happens precisely when $Q$ is root independent.
\end{proof}

%%%%%%%%%%%%%%%%%%%%%%%%%%%%%%%%%%%%%%%%%%%%%%%%%%%%%%%%%
\subsection{Stratification of the brick manifold}
%%%%%%%%%%%%%%%%%%%%%%%%%%%%%%%%%%%%%%%%%%%%%%%%%%%%%%%%%

We give a stratification whose dual, in some sense, is the subword complex. We now introduce and recall some notation. Consider a complex semisimple Lie group $G$ with upper and lower Borel subgroups $B=B^+$ and $B^-$, and Weyl group $W$. For $u\in W$ we have the \defn{Schubert cell} $\mathring{X}_v:=B^-uB/B$ and the \defn{opposite Schubert cell} $\mathring{X}^v:=B^+uB/B$. The Schubert variety $X^v$ and opposite Schubert variety $X_u$ are the closure of $\mathring{X}_u$ and $\mathring{X}^u$, respectively. Given $u,v\in W$, the \defn{open Richardson variety} is $\mathring{X}^v_u:=\mathring{X}^v\cap \mathring{X}_u$. The \defn{Richardon variety} $X^v_u$ is the closure of $\mathring{X}^v_u$ This variety is nonempty if and only if $u\leq v$ in the Bruhat order, and its dimension is $\ell(v)-\ell(u)$. Then $\displaystyle{X^v_u=\coprod_{u\leq x< y\leq v}\mathring{X}^y_x}$ is a stratification.

Given a Bott-Samelson variety $BS^Q:=(P_{i_1}\times\cdots\times P_{i_m})/B^m$ and a subword $R$ of $Q$, we can realize $BS^R$ inside $BS^Q$ by 
\begin{displaymath}BS^R=\{(p_1,\ldots,p_m) : p_{i_j}=id \text{ if } s_{i_j}\notin R \};\end{displaymath}
note that $BS^R\cap BS^S=BS^{R\cap S}$. 
Let $\mathring{BS^R_u}:=BS^R\cap m_Q^{-1}(\mathring{X}_u)$ then these subvarieties yield a stratification of $BS^Q$, where $R$ ranges over all subwords of $Q$ and $u\in W$. 
We have that $BS^R_u\neq \emptyset$ if and only if  $\Dem(R)\geq u$. Moreover, $BS^R_u\subseteq BS^S_v$ if and only if $R$ is a subword of $S$ and $u\geq v$ in Bruhat order.
This induces a stratification of $\fiber{w}$, described in the following theorem, that is dual to the subword complex $\splex{Q}{\Dem(Q)}$.

\begin{theorem}\label{strati} Let $w=\Dem(Q)$. Brick manifolds have the stratification 
\begin{displaymath}\fiber{w}=\coprod_R \mathring{BS^R_w},\end{displaymath}
 where $R$ ranges over all subwords of $Q$ with $\Dem(R)=w$. This stratifications satisfies the nice property that the intersection of any two strata, when nonempty, is again a stratum (instead of a union of strata).
\end{theorem}

\begin{proof}  If $p\in  \mathring{BS^Q_w}$, then $m_Q(p)\in X_w^w=\{wB/B\}$ and so $\fiber{w}$ is a stratum of $BS^Q$. Moreover, if $ \mathring{BS^R_w}\subset \fiber{w}$ is nonempty then $R$ is a subword and $\Dem(R)\geq u\geq w$ but then $\Dem(R)=u=w$. Therefore, the stratification of the Bott-Samelson variety restricts to a stratification of the brick manifold and $\mathring{BS^R_w}\cap \mathring{BS^S_w}=\mathring{BS^{R\cap S}_w}$.
\end{proof}

%%%%%%%%%%%%%%%%%%%%%%%%%%%%%%%%%%%%%%%%%%%%%%%%%%%%%%%%%
\subsection{Brick manifolds and Richardson varieties}
%%%%%%%%%%%%%%%%%%%%%%%%%%%%%%%%%%%%%%%%%%%%%%%%%%%%%%%%%

A subfamily of brick varieties were used before by Brion in \cite{MR2143072} in the proof of Theorem 4.2.1 as a resolution of singularities for Richardson varieties. Given a word $Q=(q_1,\ldots,q_\sizeQ)$, the \defn{opposite Bott-Samelson variety} $BS_Q$ is defined analogously to $BS^Q$. More precisely, 
\begin{displaymath}BS_Q:=(P^-_{i_1}\times\cdots\times P^-_{i_\sizeQ})/(B^-)^m,\end{displaymath}
where $B^-$ is the opposite Borel and the $P^-_i$ are the \defn{opposite minimal parabolics}. The natural map to the flag variety is 
\begin{align*} BS_Q &\lra{m_Q} G/B\\
(p_1,\ldots,p_m) &\lmp{}(p_1\cdots p_mw_0)B/B. 
\end{align*}
Given an element $u\in W$ the opposite Bott-Samelson variety $BS_Q$ is a resolution of the Schubert variety $X_u$, where $Q$ is a reduced word for $uw_0$. Given $u,v\in W$, let $R$ be a reduced word for $v$ and $T$ be a reduced word word for $uw_0$, then the fibered product $BS^R\times_{G/B}BS_T$ with the map induced by $m_R$ is Brion's resolution of the Richardson variety $X^v_u$. We will prove that this fibered product is a brick manifold.

Given $u,v\in W$, let $R$ be a reduced word for $v$ and $S$ be a reduced word for $u^{-1}w_0$, where $w_0$ is the longest word in $W$. Now, if $Q=R+S$, i.e. $Q$ is the concatenation of $R$ and $S$, and $u\leq v$ then $\Dem(Q)=w_0$. Moreover, the brick manifold $\fiber{w_0}$ together with the map to the flag in the middle gives a resolution of the Richardson variety $X^v_u$.

\begin{example} Let $R=(s_1,s_2,s_3,s_1,s_2)$ and $S=(s_3,s_1,s_2,s_1)$. Then $\fiber{w_0}$ together with the map given by the red flag is a resolution of singularities for $X^v_u$ with $v=s_1s_2s_3s_1s_2$ and $u=s_1s_2$.

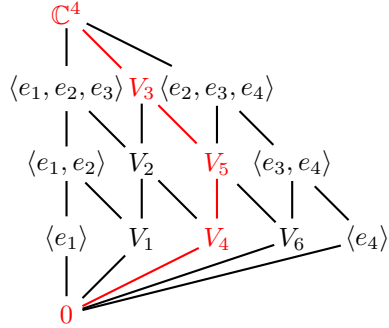
\begin{center}
  \begin{tikzpicture}
    {\color{red}\node (top) at (0,0) {$\C^4$};}
     \node [below of=top] (flag3)  {$\langle e_1,e_2,e_3\rangle$};
    \node [below of=flag3] (flag2)  {$\langle e_1,e_2\rangle$};
    \node [below of=flag2] (flag1) {$\langle e_1\rangle$};
    \node [right of=flag1] (L1) {$V_1$};
        \node [right of=flag2] (P1) {$V_2$};
               {\color{red} \node [right of=flag3] (S1) {$V_3$};}
           \node [right of=S1] (S2) {$\langle e_2,e_3,e_4\rangle$};     
        {\color{red} \node [right of=L1] (L2) {$V_4$};}
	\node [right of=L2] (L3) {$V_6$};
        {\color{red} \node [right of=P1] (P2) {$V_5$};}
        \node[right of=L3] (L4) {$\langle e_4\rangle$};
         \node [right of=P2] (P3) {$\langle e_3,e_4\rangle$};
        {\color{red}\node [below of=flag1] (down) {0};}
\draw [thick,shorten <=-2pt] (top) -- (flag3);
\draw [thick,shorten <=-2pt] (flag2) -- (flag1);
\draw [thick,shorten <=-2pt] (flag2) -- (flag3);
\draw [thick,shorten <=-2pt] (P1) -- (L1);
\draw [thick,shorten <=-2pt] (P1) -- (L2);
{\color{red}\draw [thick,shorten <=-2pt] (P2) -- (L2);}
\draw [thick,shorten <=-2pt] (S1) -- (P1);
{\color{red}\draw [thick,shorten <=-2pt] (top) -- (S1);}
\draw [thick,shorten <=-2pt] (top) -- (S2);
\draw [thick,shorten <=-2pt] (S2) -- (P2);
\draw [thick,shorten <=-2pt] (flag2) -- (L1);
\draw [thick,shorten <=-2pt] (flag1) -- (down);
\draw [thick,shorten <=-2pt] (L1) -- (down);
\draw [thick,shorten <=-2pt] (L3) -- (down);
{\color{red}\draw [thick,shorten <=-2pt] (L2) -- (down);}
{\color{red}\draw [thick,shorten <=-2pt] (S1) -- (P2);}
\draw [thick,shorten <=-2pt] (P3) -- (L3);
\draw [thick,shorten <=-2pt] (P1) -- (flag3);
\draw [thick,shorten <=-2pt] (L3) -- (P2);
\draw [thick,shorten <=-2pt] (L4) -- (P3);
\draw [thick,shorten <=-2pt] (L4) -- (down);
\draw [thick,shorten <=-2pt] (P3) -- (S2);
\end{tikzpicture}  
\end{center}
\end{example}

\begin{theorem}\label{richi} Let $u\leq v$ and $Q=R+S$, where $R$ is a reduced word for $v$ and $S$ is a reduced word for $u^{-1}w_0$. The brick manifold $\fiber{w_0}$ together with the map $m_R:BS^R_w\rightarrow G/B$ is a resolution of the singularities of the Richardson variety  $X^v_u$.
\end{theorem}

\begin{proof} Let $T$ be the reduced word for $uw_0$ obtained by taking $S^{-1}$ and conjugating each letter by $w_0$. The result follows from identifying the fibered product $BS^R\times_{G/B}BS_T$ with the brick manifold $\fiber{w_0}$. If $R=(q_1,\ldots,q_{|R|})$, then then the points in $BS^R$ consist of lists of $m+1$ flags in $G/B$
\begin{displaymath} (F_0=B/B, F_1,\ldots,F_{|R|})\in BS^R\end{displaymath}
 such that the $k$-th flag agrees with the previous one except possibly on the subspace corresponding to $q_k$, and if $F_{k-1}=gB/B$ and $F_k=hB/B$, then $h^{-1}gB/B\in X^{q_k}$. Similarly, if $T=(q_1,\ldots,q_{|T|})$, then then the points in $BS_{S^{-1}}$ consist of lists of $m+1$ flags in $G/B$
\begin{displaymath} (E_0=w_0B/B, E_1,\ldots,E_{|T|})\in BS_T\end{displaymath}
 such that the $k$-th flag agrees with the previous one except possibly on the subspace corresponding to $q_k$, and if $E_{k-1}=gB/B$ and $E_k=hB/B$, then $h^{-1}gB/B\in X^{w_0^{-1}q_kw_0}$. Therefore, the fibered product $BS^R\times_{G/B}BS_T$ consists of the lists of flags of the form
\begin{displaymath} (F_0=B/B, F_1,\ldots,F_{|R|}=E_{|T|}, E_{|T|-1},\ldots,E_0=w_0B/B)\end{displaymath}
such that consecutive flags agree in the described way, together with the maps $\displaystyle{BS^R\lra{m_R}G/B}$ and $\displaystyle{BS_T\lra{m_T}G/B}$ that map the list of flags to $F_{|R|}=E_{|T|}$. This is precisely the brick manifold $\fiber{w_0}$.
\end{proof}

\bibliographystyle{plain}
\nocite*
\bibliography{bibliography}

\end{document}